\documentclass[12pt,a4paper,twoside,reqno]{amsart}

\usepackage[OT1]{fontenc}
\usepackage{type1cm}
\usepackage[english]{babel}
\usepackage{amsthm}
\usepackage{epsfig}
\usepackage{graphicx,epsfig,subfigure}

\numberwithin{equation}{section}

\theoremstyle{definition}
\newtheorem{thm}{Theorem}[section]
\theoremstyle{definition}
\newtheorem{lm}[thm]{Lemma}
\theoremstyle{definition}
\newtheorem{cor}[thm]{Corollary}
\theoremstyle{definition}

\theoremstyle{definition}
\newtheorem{df}[thm]{Definition}

\theoremstyle{definition}
\newtheorem*{osc}{\textbf{Open set condition (OSC)}}

\theoremstyle{definition}
\newtheorem*{bdp}{Bounded distortion property (BDP)}

\theoremstyle{remark}
\newtheorem*{prpr}{}

\theoremstyle{remark}
\newtheorem{rem}[thm]{Remark}

\newcommand{\R}{\mathbb{R}}
\newcommand{\Rn}{\mathbb{R}^{n}}

\newcommand{\N}{\mathbb{N}}

\newcommand {\lesssim} {\ {\raise-.5ex\hbox{$\buildrel<\over\sim$}}\ }

\renewcommand{\eta}{\beta}

\begin{document}
\title[\tiny directed porosity on CIFS and weak convergence of singuar integrals]
{Directed Porosity on Conformal Iterated Function systems and Weak
convergence of Singular Integrals }
\author{Vasilis Chousionis}

\thanks{The author is supported by the Finnish Graduate School in
Mathematical Analysis.} 

\begin{abstract} The aim of the present paper is twofold. We study
directed porosity in connection with conformal iterated function
systems (CIFS) and with singular integrals. We prove that limit sets
of finite CIFS are porous in a stronger sense than already known.
Furthermore we use directed porosity to establish that truncated
singular integral operators, with respect to general Radon measures
$\mu$ and kernels $K$, converge weakly in some dense subspaces of
$L^2(\mu)$ when the support of $\mu$ belongs to a broad family of
sets. This class contains many fractal sets like CIFS's limit sets.
\end{abstract}

\maketitle
\section{Introduction}
A set $E \subset \Rn$ is called porous, or uniformly lower porous,
if there exists a constant $c>0$ so that for each $x \in E$ and $0
<r < d(E)$ there exists $y \in B(x,r)$ satisfying
$$B(y,cr) \subset B(x,r) \setminus E.$$
Here $B(x,r)$ is the closed ball centered at $x$ with radius $r$ and
$d(\cdot)$ denotes diameter. Dimensional properties of porous sets
were studied by Mattila in \cite{M1}. Motivated by his work
different aspects of porosity have been investigated widely in
relation with dimensional estimates and densities. See e.g.
\cite{S1}, \cite{KS1}, \cite{KS2} and \cite{JJKV}. Some other
applications of porosities related with the boundary behavior of
quasiconformal mappings can be found in \cite{KR}, \cite{MVu} and
\cite{V1}.

Questions regarding porosities arise naturally in fractal geometry.
This can be understood heuristically since many familiar self
similar sets in $\Rn$ are constructed by removing pieces out of some
$n$-dimensional set in every step of the iteration process. The
theory of conformal iterated function systems (CIFS), where the
limit set is generated by uniformly contracting conformal maps, was
studied systematically by Mauldin and Urba\'{n}ski in \cite{MU}.
This theory extends previous results and allows one to analyze many
more limit sets than the ones emerging from the usual similitude
iterated function systems. The precise assumptions on CIFS are given
in Section 2.

Over the past several years many authors have studied the dynamic
and geometric properties of such limits sets, porosity being one of
them. See e.g. \cite{MMU}, \cite{VMU}, \cite{U} and \cite{K}. In
\cite{U}, Urba\'{n}ski gave necessary and sufficient conditions for
the limit set of a CIFS on $\R^n$ to be porous. As a consequence if
the CIFS is finite and its limit set has Hausdorff dimension less
than $n$, it is also porous. Furthermore in the aforementioned paper
some interesting applications of porosities in continued fractions
were established.

If one considers typical examples of $(n-1)$-dimensional CIFS's
limit sets, for example very simple self similar sets like the four
corners Cantor set in the plane, intuitively one expects to find
holes spread in many directions. Motivated by this simple
observation we introduce the notion of directed porous sets. For $m
\in \N, 0 <m <n,$ we denote by $G(n,m)$ the set of all
$m$-dimensional planes in $\R^n$ crossing the origin.
\begin{df}
\label{mdf}Suppose $V\in G(n,m)$. A set $E\subset \mathbf{R}^{n}$
will be called \emph{$V$-directed porous at $x \in E$}, if there
exists a constant $c(V)_x>0$, such that for all $r>0$ we can find $y\in
V+x$ satisfying
\begin{equation*}
B(y,c(V)_x r) \subset B(x,r) \setminus E.
\end{equation*}
If $E$ is $V$-directed porous at every $x \in E$, and
$c(V)=\inf\{\sup c(V)_x:x\in E\}>0$, it will be called \emph{$V$-directed
porous}.
\end{df}
Recall that a set $E\subset \mathbb{R}^{n}$ will be called
$m$-rectifiable for $m=1,..,n$, if there exist $m$-dimensional
$C^1$-submanifolds $M_{i}$, $i\in \N$, such that
\begin{equation*}
\mathcal{H}^{m}(E \setminus \bigcup_{i=1}^\infty M_{i})=0.
\end{equation*}
Here $\mathcal{H}^m$ denotes the $m$-dimensional Hausdorff measure.
Sets intersecting $m$-rectifiable sets in a set of zero
$\mathcal{H}^m$ measure are called $m$-purely unrectifiable. More
information about rectifiability and related topics can be found in
\cite{M2}.

In Section 2, we show that limit sets of finite CIFS have very
strong porosity properties, extending Urba\'{n}ski's result in the
following sense.
\begin{thm}
\label{mpthm} Let $E\subset \mathbb{R}^{n},n \geq 2,$ be the limit
set of a given finite CIFS. If $E$ is $m$-purely unrectifiable then
it is $V$-directed porous for all $V\in G(n,m)$.
\end{thm}
In \cite{K}, K\"{a}enm\"{a}ki studied the geometric structure of
CIFS's limit sets. He proved that if $E$ is a limit set of a given
CIFS with $\textmd{dim}_{\mathcal{H}}E=t$, where
$\textmd{dim}_{\mathcal{H}}$ stands for Hausdorff dimension, and
$l\in \N,0<l<n,$ then either
\begin{enumerate}
\item $\mathcal{H}^t(E \cap M)=0$ for every $l$-dimensional
$C^1$-submanifold of $\Rn$, or,
\item $E$ lies in some $l$-dimensional affine
subspace or $l$-dimensional geometric sphere when $n>2$, and in some
analytic curve when $n=2$.
\end{enumerate}
Combining the previous rigidity result with Theorem \ref{mpthm} we
derive the following corollary.
\begin{cor}
\label{mcor} Let $E\subset \mathbb{R}^{n},n\geq2,$ be the limit set
of a given finite CIFS. If $\textmd{dim}_{\mathcal{H}}E\leq m$ where
$m \in \N, 0 <m <n,$ then $E$ is $V$-directed porous at every $x \in
E$ for all, except at most one, $V\in G(n,m)$.
\end{cor}

The motivation for this paper comes from the theory of singular
integral operators with respect to general measures. Given a Radon
measure $\mu$ on $\Rn$ and a $\mu$-measurable kernel
$K:\mathbb{R}^{n}\setminus \{0\}\rightarrow \mathbb{R}$ that
satisfies the antisymmetry condition
\begin{equation*}K(-x)=-K(x)\text { for
all }x\in\Rn,
\end{equation*}the singular integral operator $T$
associated with $K$ and $\mu$ is formally given by
\begin{equation*}
T^{\mu ,K}(f)(x)=\int K(x-y)f(y)d\mu y.
\end{equation*}
Since the above integral does not usually exist when $x\in
\textmd{spt} \mu$, the truncated singular integral operators $T^{\mu
,K}_{\varepsilon},\varepsilon >0$;
\begin{equation*}
T^{\mu ,K}_{\varepsilon}(f)(x)=\int_{|x-y|>\varepsilon}
K(x-y)f(y)d\mu y,
\end{equation*}
are considered. Often for simplicity we will denote $T^{\mu
,K}_{\varepsilon}$ by $T_{\varepsilon}$. Using this convention one
defines the maximal operator $T^*$,
\begin{equation*}
T^*(f)(x)=\underset{\varepsilon > 0}{\sup }|T_{\varepsilon}(f)(x)|,
\end{equation*}
and the principal values of $T(f)$ at every $x \in \Rn$ which, if
they exist, are given by
\begin{equation*}
\textmd{p.v.}T(f)(x)=\underset{\varepsilon \rightarrow 0}{\lim
}T_{\varepsilon}(f)(x).
\end{equation*}
In the classical setting, when $\mu=\mathcal{L}^n$, the Lebesgue
measure in $\Rn$, and $K$ is a standard Calder\'{o}n-Zygmund kernel,
cancelations and the denseness of smooth functions in $L^1$ force
the principal values to exist almost everywhere for $L^1$-functions.
One could naturally ask if the $L^2(\mu)$-boundedness of $T^*$,
which means that there exists a constant $C>0$ such that for all
$f\in L^2(\mu)$,
\begin{equation*}
\int {T^*(f)}^2 d\mu \leq C\int {|f|}^2d\mu,
\end{equation*}
forces the principal values to exist. The answer to the above
question is not always positive, see e.g. \cite{D3} and \cite{Cs}.
Interestingly enough even when $\mu$ is an $m$-dimensional
Ahlfors-David (AD) regular measure in $\Rn$:
\begin{equation*}
 C^{-1}r^{m}\leq \mu (B(x,r))\leq Cr^{m} \text{ for }x \in
 \textmd{spt}\mu, 0 <r<\textmd{diam}(\textmd{spt}\mu),
\end{equation*}
and $K$ is any of the coordinate Riesz kernels:
\begin{equation*}
R_i^{m} (x)=\frac {x_i}{|x|^{m+1}}\text{ for } i=1,...,n,
\end{equation*}
the question remains open for $m>1$. For $m=1$, it has positive
answer by Tolsa, see \cite{T1}, even for more general measures.
Previous results by Mattila, Melnikov and Verdera, see \cite{MM} and
\cite{MMV}, dealt with the affirmative in the case of AD-regular
measures.

Recently, in \cite{MV}, Mattila and Verdera proved that, for general
measures and kernels, the $L^2(\mu)$-boundedness of $T^*$ implies
that the operators $T_{\varepsilon}$  converge weakly in $L^2(\mu)$.
This means that there exists a bounded linear operator
$T:L^2(\mu)\rightarrow L^2(\mu)$ such that for all $f,g \in
L^2(\mu)$,
\begin{equation}
\label{wcmv}
\lim _{\varepsilon \rightarrow 0} \int T_{\varepsilon
}(f)(x)g(x)d\mu x=\int T(f)(x)g(x)d\mu x.
\end{equation}
Furthermore they showed that
\begin{equation}
\label{avco}T(f)(z)= \lim _{r \rightarrow 0} \frac{1}{\mu(B(z,r)}
\int_{B(z,r)} \int_{\Rn \setminus B(z,r)} K(x-y)f(y)d\mu yd\mu x
\end{equation}
for $\mu$ a.e. $z$. One of the main points in their proof is that
$L^2(\mu)$-boundedness forces the limits
\begin{equation}
\label{wc}
\lim _{\varepsilon \rightarrow 0} \int T_{\varepsilon
}(f)(x)g(x)d\mu
\end{equation} to exist when $f,g$ are finite
linear combinations of characteristic functions of balls. We will
denote this dense subspace of $L^2(\mu)$ by
$\mathcal{X}_{B}(\mathbb{R}^{n})$.

Recall that if $E$ is a $\mathcal{H}^m$-measurable set with
$\mathcal{H}^m(E) <\infty$ and $\mu=\mathcal{H}^m \lfloor E $, the
restriction of $\mathcal{H}^m$ on $E$, by the works of Mattila and
Preiss \cite{MPr}, Mattila and Melnikov \cite{MM}, Verdera \cite{Ve}
and Tolsa \cite{T2}, the principal values
\begin{equation*}
\lim_{\varepsilon \rightarrow 0}\int_{\mathbb{R}^{n}\setminus
B(x,\varepsilon )}R_i^{m} (x-y)d\mu y
\end{equation*}
exist $\mu$ almost everywhere if and only if the set E is
$m$-rectifiable.

With the last two paragraphs in mind one might ask if weak limits
like in (\ref{wc}) might exist if we remove the strong
$L^2$-boundedness assumption even when the measures are supported in
some purely unrectifiable sets. Before stating the main results of
Section 3 we give some basic notation. Let
\begin{equation}
\label{qcu} Q(\mathbb{R}^{n})=\{A(x,r):x\in \mathbb{R}^{n},r>0
\text{ and } A(x,r)=\prod_{i=1}^{n}[x^{i}-r/2,x^{i}+r/2)\}
\end{equation}
and denote by $\mathcal{X}_{Q}(\mathbb{R}^{n})$ the dense subspace
of $L^2(\mu)$, in the same manner as
$\mathcal{X}_{B}(\mathbb{R}^{n})$, while instead of balls we take
cubes from $Q(\mathbb{R}^{n})$.
\begin{thm}
\label{msithm} Let $\mu$ be a finite Radon measure on $\Rn,n\geq2,$
satisfying
\begin{equation}
\label{mg} \mu(B(x,r))\leq Cr^{n-1}\text{ for all }x\in
\textmd{spt}\mu \text{ and }r>0.
\end{equation} Let $K:\mathbb{R}^{n}\setminus
\{0\}\rightarrow \mathbb{R}$ be an antisymmetric kernel, satisfying
for all $x \in \Rn$,
\begin{equation}
\label{kg} \left| K(x)\right|  \leq C_{K}\left| x\right| ^{-(n-1)},
\end{equation}
where $C_{K}$ is a constant depending on the kernel $K$.
\begin{enumerate}
\item If $\textmd{spt}\mu$ is $V^i$-directed porous for $i=1,..,n$, where
$V^{i}=\left\{ x\in \mathbb{R}^{n}:x^{i}=0\right\}$ are the usual
coordinate planes of $\Rn$, the truncated singular integral
operators $T^{\mu ,K}_{\varepsilon}$ converge weakly in $\mathcal{X
}_{Q}(\mathbb{R}^{n})$.
\item If $\textmd{spt}\mu$ is $V$-directed porous for all $V\in G(n,n-1)$, the truncated
singular integral operators $T^{\mu ,K}_{\varepsilon}$ converge
weakly also in $\mathcal{X }_{B}(\mathbb{R}^{n})$.
\end{enumerate}
\end{thm}
As an immediate consequence of Theorems \ref{mpthm} and \ref{msithm}
we obtain the following corollary.
\begin{cor}
\label{scor} Let $E \subset \Rn,n \geq 2,$ be a $(n-1)$-purely
unrectifiable limit set of a given finite CIFS. If $\mu
=\mathcal{H}^{n-1}\lfloor E$ and $K:\mathbb{R}^{n}\setminus
\{0\}\rightarrow \mathbb{R}$ is a kernel as in Theorem \ref{msithm},
the limits
\begin{equation*}
\lim _{\varepsilon \rightarrow 0} \int T_{\varepsilon
}(f)(x)g(x)d\mu
\end{equation*}
exist for $f,g \in \mathcal{X }_{Q}(\mathbb{R}^{n})$ and $f,g \in
\mathcal{X }_{B}(\mathbb{R}^{n})$.
\end{cor}
We conclude the introductory part with the following two remarks.
\begin{rem} The kernels satisfying the assumptions of
Theorem \ref{msithm} belong to a quite broad class,
$(n-1)$-dimensional Riesz kernels being one representative. Notice
that we do not even require them to be continuous. In \cite{CM}, it
was proved, with different techniques, that weak convergence in
$\mathcal{X }_{Q}(\mathbb{R}^{n})$ and in $\mathcal{X
}_{B}(\mathbb{R}^{n})$ holds for much more general measures if we
restrict the kernels to a smaller but still large and widely used
family.
\end{rem}
\begin{rem} One
cannot hope of replacing the function spaces
$\mathcal{X}_{B}(\mathbb{R}^{n})$ and
$\mathcal{X}_{Q}(\mathbb{R}^{n})$ with $L^2(\mu)$ in Theorem
\ref{msithm}. This follows because as it was remarked in \cite{MV},
by the Banach-Steinhaus Theorem, the weak convergence in $L^2(\mu)$
implies that the operators $ T_{\varepsilon}$ are uniformly bounded
in $L^2(\mu)$ and singular integral operators associated with
$1$-dimensional Riesz kernels and $1$-purely unrectifiable measures
are not bounded in $L^2(\mu)$.
\end{rem}

\section{Directed porosity on Conformal Iterated Function Systems}
We begin by describing the setting of CIFS, as introduced in
\cite{MU}. Let $I$ be a countable set with at least two elements and
let
\begin{equation*}
I^{\ast }=\underset{m\geq 1}{\bigcup }I^{m}\text{ and }I^{\infty
}=I^{\mathbb{N}}
\end{equation*}
If $w=(i_{1},i_{2},..)\in I^{\ast }\cup I^\infty$ and $n \in \N$,
does not exceed $|w|$, the length of $w$, we denote
$w|_{n}=(i_{1},..,i_{n})$.

Choose $\Omega $ to be some open, bounded and connected subset of
$\mathbb{R}^{n}$ and let $\{\varphi _{i}\}_{i\in I}$, $\varphi
_{i}:\Omega \rightarrow \Omega$, be a family of injective
maps such that for every $i\in I$ there exists some $0<s_{i}<1$ such
that
\begin{equation}
\label{cont} \left| \varphi _{i}(x)-\varphi _{i}(y)\right| \leq
s_{i}\left| x-y\right|.
\end{equation}
Functions satisfying (\ref{cont}) are called contractive. We will
further assume that the mappings $\varphi _{i}$ are uniformly
contractive, that is, $s=\sup\{s_i:i \in I\} <1$ and conformal. Conformality here
stands for $|\varphi_i'|^n=|J\varphi_i|$, where $J$ is the Jacobian
and the norm in the left side is the usual ``sup-norm'' for linear
mappings. This definition is usually referred as
$1$-quasiconformality, see e.g. \cite{V1}. By Theorem 4.1 of
\cite{R} conformal maps on subsets of $\mathbb{R}^{n},n\geq2,$ are
$C^\infty$. Assume also that there exists a compact set $X\subset
\Omega $ such that $\textmd{int}(X)\neq \emptyset $ with the
property that $\varphi _{i}(X)\subset X$ for all $i\in I$. Notice
that for $\Omega=\R^n, n\geq3,$ conformal, contractive mappings are
similitudes, which means that equality holds in (\ref{cont}). We
will call a family of functions $\{\varphi_i\}_{i\in I}$, as
described above, a \emph{conformal iterated function system} (CIFS)
if it satisfies the following property.
\begin{osc}
There exists a non-empty open set $U\subset X$ (in the relative
$X-$topology) such that $\varphi _{i}(U)\subset U$ for every $i\in
I$ and $\varphi _{i}(U)\cap \varphi _{j}(U)= \emptyset $ for
every pair $i \neq j\in I$.
\end{osc}

For $w=(i_{1},..,i_{m})\in I^{m}$, denote $\varphi _{w}=\varphi
_{i_{1}}\circ ..\circ \varphi _{i_{m}}$ and notice that
\begin{equation*}
d(\varphi _{w}(X))\leq s^{m}d(X).
\end{equation*}
Now define the mapping $\pi :I^{\infty }\rightarrow X$ such that
\begin{equation*}
\pi (w)=\underset{m\geq 1}{\bigcap }\varphi _{w|_{m}}(X).
\end{equation*}
The \emph{limit set} of the CIFS is defined as,
\begin{equation*}
E=\pi (I^{\infty })=\underset{w\in I^{\infty }}{\bigcup }\underset{m\geq 1}{%
\bigcap }\varphi _{w|_{m}}(X).
\end{equation*}

We will be interested in finite CIFS, where $\Omega \subset
\mathbb{R}^{n},n\geq 2$. The following important property of these
function systems follows from smoothness of the mappings $\varphi _{i}$, for a
proof see \cite{MU}, Lemma 2.2.
\begin{bdp}
There exists some $K\geq 1$ such that
\begin{equation*}
| \varphi _{w}^{^{\prime }}(x)| \leq K| \varphi _{w}^{^{\prime
}}(y)| \text{ for }w\in I^{\ast }\text{ and }x,y\in \Omega,
\end{equation*}
\end{bdp}

Finally we state two properties of CIFS that are going to be used
often in the proofs. In both properties constants depend only on the
initial CIFS parameters. The first one is a direct consequence of
BDP and the connectedness of $\Omega$. Since finite CIFS are
controlled Moran constructions, it follows by \cite{KV} that (CIFS
2) is equivalent to the OSC.

\begin{prpr}[\textbf{CIFS 1}]
\label{cifs1} There exists some constant $D \geq 1$ such that
\begin{equation*}
D^{-1}\| \varphi _{w}^{^{\prime }}\| \leq d(\varphi _{w}(E))\leq D\|
\varphi _{w}^{^{\prime }}\| \text{ for }w\in I^{\ast }.
\end{equation*}
Here $\|\varphi'_w\|=\textmd{sup}_{x\in \Omega}|\varphi_w'(x)|$.
\end{prpr}
\begin{prpr}[\textbf{CIFS 2}]
\label{cifs2} Denote
$$I(x,r)=\{w \in I^{\ast }:\varphi_w(E) \cap B(x,r)\neq 0 \text{ and }
d(\varphi_w(E))\leq r < d(\varphi_{w|_{|w|-1}}(E))\},$$where
$\varphi_0=id$. There exist a positive number $N\in \mathbb{N}$ and
a constant $C>0$, such that for every $x\in \mathbb{R}^{n}$ and
every $ 0<r\leq1$
\begin{enumerate}
\item $\textmd{card}(I(x,r))\leq N$, where $\textmd{card}(\cdot)$ denotes
cardinality,
\item $Cr\leq d(\varphi _{w}(E))\leq r$ for $w\in I(x,r)$,
\item $E\cap B(x,r)\subset \underset{w\in I(x,r)}{\bigcup }\varphi
_{w}(E)$.
\end{enumerate}
\end{prpr}
The main result of this section reads as follows.
\begin{thm}
\label{cthm} Let $E\subset \mathbb{R}^{n},n\geq 2,$ be the limit set
of a given finite CIFS such that every conformal map $F:\Omega
\rightarrow \Rn$ satisfies
\begin{equation}
\label{conf}
 F(\Omega \cap B(x,r)\cap (V+x))\cap E^c \neq \emptyset\text{ for all
}x\in \Rn, r>0 \text{ and } V\in G(n,m).
\end{equation}
Then $E$ is $V$-directed porous for all $V\in G(n,m)$.
\end{thm}
Notice that Theorem \ref{mpthm} follows immediately from Theorem
\ref{cthm} since $m$-purely unrectifiable sets satisfy (\ref{conf}).
The main step in proving Theorem \ref{cthm} is the following Lemma.
\begin{lm}
\label{confim} Let $E\subset \mathbb{R}^{n}$ be the limit set of a
given CIFS such that (\ref{conf}) holds for every conformal map
$F:\Omega \rightarrow \Rn$. Then for every $V\in G(n,m)$ and every
$\beta>0$ there exists some $a(\beta)>0$ such that for every $x\in
\mathbb{R}^{n}, 0<r\leq 1,w\in I(x,r),y\in x+V$ and $s\geq \beta
d(\varphi _{w}(E))$ satisfying
\begin{equation*}
B(y,s)\subset B(x,r),
\end{equation*}
there exists $z\in x+V$ \ and $l\geq a(\beta)s$ such that
\begin{equation*}
B(z,l)\subset B(y,s)\backslash \varphi _{w}(E).
\end{equation*}
\end{lm}

\begin{proof}Without loss of generality assume that $E \subset B(0,1)$.
We will prove Lemma \ref{confim} in the case where $V$ is some
$m$-coordinate plane, say $V=\{x\in \mathbb{R}^{n}:x^{i}=0$ for
$i=m+1,..,n\}$. The general statement follows after appropriate
rotations of the set E. Let $V_{x}=x+V$ for $x \in \mathbb{R}^{n}$.
By way of contradiction, suppose that Lemma \ref{confim} does not
hold. Then there exists some constant $\eta
>0$ such that for every $j \in \N $ there exist sequences
\begin{eqnarray*}
\{x_{j}\}_{j\in \N} &\in &B(0,1),\\
\{r_{j}\}_{j\in \N}&\in & (0,1], \\
\{w_{j}\}_{j\in \N} &\in &I^{\ast \text{ }}\text{ such that
}w_{j}\in
I(x_{j},r_{j})\text{ for every }j\in \mathbb{N} ,\\
\{y_{j}\}_{j\in \N} &\in &B(0,1)\cap V_{x_{j}},\\
\{s_{j}\}_{j\in \N}&\in & (0,1],
\end{eqnarray*}
satisfying for all $j \in \N$ the following three conditions.
\begin{description}
\item[C1] $B(y_{j},s_{j})\subset B(x_{j},r_{j})$.

\item[C2] $ s_{j} \geq \eta d(\varphi _{w_{j}}(E))$.

\item[C3] For every $z\in V_{x_{j}}$ the condition
\begin{equation*}B(z,l)\subset B(y_{j},s_{j})\backslash \varphi _{w_{j}}(E)
\end{equation*}
implies $l<\frac{1}{j}s_{j}$.
\end{description}
By passing to an appropriate subsequence, if necessary, we find
$y\in B(0,1)$ such that
\begin{equation*}
y_{j}\rightarrow  y.
\end{equation*}
From now on we will denote $V_{x_{j}}=V_{y_{j}}$ by $V_{j}$. Let
$\Psi _{j}:\mathbb{R}^{n}\rightarrow \mathbb{R}^{n}$ be defined for
$z\in \mathbb{R}^{n}$ as,
\begin{equation*}
\Psi _{j}(z)=\| \varphi _{w_{j}}^{^{\prime }}\|
^{-1}(z-y_{j})+y_{j}.
\end{equation*}
We are going to use the following properties of $\Psi _{j}$:

\begin{description}
\item[$\Psi1$] For all pairs $z,w\in \mathbb{R}^{n}$
\begin{equation*}
| \Psi _{j}(w)-\Psi _{j}(z)| = \| \varphi _{w_{j}}^{^{\prime }}\|
^{-1}| w-z|.
\end{equation*}

\item[$\Psi2$] For every $\delta >0$, and $V_{j}(\delta )=\{x\in \R^n: d(x,V_j)
<\delta\}$,
\begin{equation*}
\Psi _{j}(V_{j})=V_{j}\text{ and }\Psi _{j}(V_{j}(\delta
))=V_{j}(\delta \| \varphi _{w_{j}}^{^{\prime }}\| ^{-1}).
\end{equation*}
\item[$\Psi3$] For every $r>0$ and every $z\in V_{j}$,
\begin{equation*}
\Psi _{j}(B(z,r))=B(\Psi _{j}(z),\| \varphi _{w_{j}}^{^{\prime }}\|
^{-1}r).
\end{equation*}
\end{description}
Denote for $j \in \N$,
\begin{equation}
\label{pj} P_{j}=V_{j}(2s_{j}j^{-1})\cap \varphi _{w_{j}}(E)\cap
B(y_{j},s_{j})
\end{equation}
and
\begin{equation}
T_{j}=\Psi _{j}(P_{j}).
\end{equation}
By (C3), for every $z\in V_{j}\cap B(y_{j},s_{j})$
\begin{equation}
\label{gap} B(y_{j},s_{j})\cap B(z,2s_{j}j^{-1})\cap \varphi
_{w_{j}}(E)\neq \emptyset.
\end{equation}
Using (\ref{gap}) we can also show that for all $q\in V_{j}\cap
B(y_{j},\| \varphi _{w_{j}}^{^{\prime }}\| ^{-1}s_{j})$ and every
$r\geq 2\| \varphi _{w_{j}}^{^{\prime }}\| ^{-1}{j}^{-1}s_{j}$,
\begin{equation}\label{int2}
B(q,r)\cap T_{j}\neq \emptyset.
\end{equation}
To see this, let
\begin{equation*}\widetilde{q}=(\| \varphi
_{w_{j}}^{^{\prime }}\| (q^{1}-y_{j}^{1})+y_{j}^{1},..,\| \varphi
_{w_{j}}^{^{\prime }}\|
(q^{m}-y_{j}^{m})+y_{j}^{m},y_{j}^{m+1},..,y_{j}^{n}),
\end{equation*}
where $q=(q^{1},..,q^{m},y_{j}^{m+1},..,y_{j}^{n})\in V_{j}\cap
B(y_{j},\| \varphi _{w_{j}}^{^{\prime }}\| ^{-1}s_{j})$. Then $\Psi
_{j}(\widetilde{q})=q$ and for $i=1,..,m$,
\begin{eqnarray*}
| \widetilde{q}^{i}-y_{j}^{i}|  =\| \varphi _{w_{j}}^{^{\prime }}\|
| q^{i}-y_{j}^{i}| \leq \| \varphi _{w_{j}}^{^{\prime }}\| \|
\varphi _{w_{j}}^{^{\prime }}\| ^{-1}s_{j}.
\end{eqnarray*}
This implies that $\widetilde{q}\in V_{j}\cap B(y_{j},s_{j})$.
Therefore, by (\ref{gap}), we get
\begin{equation*}
B(y_{j},s_{j})\cap B(\widetilde{q},2s_{j}j^{-1})\cap \varphi
_{w_{j}}(E)\neq \emptyset.
\end{equation*}
Consequently
\begin{equation*}
\Psi _{j}(V_{j}(2s_{j}j^{-1})\cap B(y_{j},s_{j})\cap \varphi
_{w_{j}}(E)\cap B(\widetilde{q},2s_{j}j^{-1}) )) \neq \emptyset
\end{equation*}
and by ($\Psi3$)
\begin{equation*}
B(q,2\| \varphi _{w_{j}}^{^{\prime }}\| ^{-1}s_{j}j^{-1})\cap \Psi
_{j}(P_{j}) \neq \emptyset.
\end{equation*}
Hence
\begin{equation*}
B(q,r)\cap T_{j} \neq \emptyset \text{ for }r\geq 2\| \varphi
_{w_{j}}^{^{\prime }}\| ^{-1}j^{-1}s_{j}.
\end{equation*}

Next we will show that there exists some constant $B>0$ such that
for every $j\in \mathbb{N}$, large enough,
\begin{equation}
\label{tjbound} B^{-1}\leq d(T_{j})\leq B.
\end{equation}
To prove (\ref{tjbound}) let $ p_{j,}q_{j}\in V_{j}\cap
B(y_{j},s_{j})$ such that
\begin{equation*}
p_{j}=(y_{j}^{1}-(s_{j}-s_{j}j^{-1}),y_{j}^{2},..,y_{j}^{n})
\end{equation*}
and
\begin{equation*}
q_{j}=(y_{j}^{1}+(s_{j}-s_{j}j^{-1}),y_{j}^{2},..,y_{j}^{n}).
\end{equation*}
Recalling (\ref{gap}) we notice that for every
\begin{equation*}
e\in B(y_{j},s_{j})\cap B(p_{j}, 2s_{j}j^{-1})\cap \varphi
_{w_{j}}(E)
\end{equation*}
and
\begin{equation*}
d\in B(y_{j},s_{j})\cap B(q_{j}, 2s_{j}j^{-1})\cap \varphi
_{w_{j}}(E),
\end{equation*}
we have
\begin{equation*}
| e-d| \geq |p_j-q_j|-|p_j-e|-|q_j-d|\geq 2s_j-6s_j j^{-1}\geq
\frac{s_{j}}{2},
\end{equation*}for $j\geq4$. Hence
\begin{equation*}
d(P_{j})=d(V_{j}(2s_{j}j^{-1})\cap \varphi _{w_{j}}(E)\cap
B(y_{j},s_{j}))\geq \frac{s_{j}}{2} \text{ where }j\geq4.
\end{equation*}
By (C2) we also deduce that
\begin{equation*}
d(P_{j})\leq d(\varphi _{w_{j}}(E))\leq \eta ^{-1}s_{j}.
\end{equation*}
Combining the two previous estimates we derive
\begin{equation}
\label{pes} \frac{s_{j}}{2}\leq d(P_{j})\leq \eta ^{-1}s_{j}.
\end{equation}
Now by (\ref{pes}), (C2) and (CIFS 1) it follows that
\begin{eqnarray*}
d(T_{j}) &=&d(\Psi _{j}(P_{j}))=\| \varphi _{w_{j}}^{^{\prime
}}\| ^{-1}d(P_{j}) \\
&\geq &\| \varphi _{w_{j}}^{^{\prime }}\| ^{-1}\frac{s_{j}}{2}%
\geq \frac{\eta }{2}\| \varphi _{w_{j}}^{^{\prime }}\|
^{-1}d(\varphi _{w_{j}}(E)) \\
&\geq &\frac{\eta }{2}D^{-1}\| \varphi _{w_{j}}^{^{\prime }}\|
^{-1}\| \varphi _{w_{j}}^{^{\prime }}\|
\end{eqnarray*}
and, by (CIFS 1),
\begin{equation*}
d(T_{j})=\| \varphi _{w_{j}}^{^{\prime }}\| ^{-1}d(P_{j})\leq \|
\varphi _{w_{j}}^{^{\prime }}\| ^{-1}d(\varphi _{w_{j}}(E))\leq D.
\end{equation*}
Therefore for all $j \in \mathbb{N},j\geq 4,$
\begin{equation*}
B^{-1}\leq d(T_{j})\leq B
\end{equation*}
where $B=\max \{D,2\eta ^{-1}D\}$. The following fact follows
immediately from (CIFS 1), (C2) and (\ref{pes}), since $P_j \subset
\varphi _{w_{j}}(E)$. We state it separately for the convenience of
the reader. For all $j \in \mathbb{N},j\geq 4,$
\begin{equation}
\label{f} \eta D^{-1}\| \varphi _{w_{j}}^{^{\prime }}\| \leq
s_{j}\leq 2D\| \varphi _{w_{j}}^{^{\prime }}\|.
\end{equation}

For every $j\in \mathbb{N}$ the functions $F_{j}:\Omega \rightarrow
\mathbb{R}^{n}$ are defined as
\begin{equation*}
F_{j}:=\Psi _{j}\circ \varphi _{w_{j}}.
\end{equation*}
Observe that for all $j \in \N$
\begin{description}
\item [F1]$F_{j}$ are conformal,

\item[F2] $F_{j}$ are bi-Lipschitz  with constants not depending on $j$.
\end{description}
Property (F2) follows from BDP and the mean value theorem. To see
this, for all $z,w\in \Omega$,
\begin{eqnarray*}
K^{-1}|z-w| &\leq &\Vert \varphi _{w_{j}}^{^{\prime }}\Vert
^{-1}\Vert
(\varphi _{w_{j}}^{^{-1}})^{^{\prime }}\Vert ^{-1}|z-w| \\
&\leq &\Vert \varphi _{w_{j}}^{^{\prime }}\Vert ^{-1}|\varphi
_{w_{j}}(z)-\varphi _{w_{j}}(w)| \\
&=&|F_{j}(z)-F_{j}(w)| \\
&\leq &|z-w|.
\end{eqnarray*}
Using the Ascoli-Arzela theorem we are now able to find some
uniformly convergent subsequence of $F_{j}$, which for the sake of
simplicity we will keep on denoting by $F_{j}$, such that
\begin{equation*}
F_{j}\rightarrow F\text{ and }F:\Omega \rightarrow \mathbb{R}^{n}
\text{ is conformal and bi-Lipschitz}.
\end{equation*}
Notice that by standard complex analysis when $n=2$, and basic
properties of M\"{o}bius maps for $n \geq 3$, it follows that the
map $F^{-1}:$ $\mathbb{R}^{n}\rightarrow \Omega $ is also conformal.

Now define
\begin{equation*}
\mathcal{G}=\{\alpha :\mathbb{N\rightarrow }\overset{\infty }{\underset{j=1}{%
\bigcup }}T_{j}\text{ such that }\alpha (j)\in T_{j}\text{ for all
}j\in \mathbb{N}\}.
\end{equation*}
and,
\begin{multline*}
T=\{t\in \mathbb{R}^{n}:\text{ there exist increasing } k:\N
\rightarrow \N  \text{ and } \\ \alpha \in \mathcal{G}\text{ such
that } \alpha (k(j))\rightarrow t\}.
\end{multline*}
The set $T$ has the following properties,

\begin{description}
\item[T1] $y\in T$.

Recall that $y$ is the limit of the sequence $y_j$. By
$(\ref{int2})$,
\begin{equation*}
B(y_{j},2\Vert \varphi _{w_{j}}^{^{\prime }}\Vert
^{-1}s_{j}j^{-1})\cap T_{j}\neq \emptyset\text{ for all }j\in
\mathbb{N}.
\end{equation*}
Therefore, by (\ref{f}), there exists some sequence
$\{t_{j}\}_{j\geq 4}$ such that for all $j\in \mathbb{N}, j\geq 4,$
\begin{equation*}
t_{j}\in T_{j}\cap B(y_{j},4Dj^{-1}).
\end{equation*}
Since $y_{j}\rightarrow y$, we also get $t_{j}\rightarrow y$ and
consequently $ y\in T$.

\item[T2] $B(y,D^{-1}\dfrac{\eta }{100})\cap V_{y}\subset T$.

Suppose that there exists some $a \in B(y,D^{-1}\dfrac{\eta
}{100})\cap V_{y}$ such that $a\notin T$. Then there exist
$r_{0}<D^{-1}\dfrac{\eta }{ 100}$ and $j_{0}\in \mathbb{N}$ such
that for all $j\geq j_{0}$,
\begin{equation*}
B(a,r_{0})\cap T_{j}=\emptyset.
\end{equation*}
Now choose some $j_{1}\in \mathbb{N}$ such that for all $j\geq
j_{1}$,
\begin{equation*}
|y_{j}-y|\leq D^{-1}\dfrac{\eta }{100}.
\end{equation*}
Then for all such $j$,
\begin{equation*}
B(a,r_{0})\subset B(y_{j},\Vert \varphi _{w_{j}}^{^{\prime }}\Vert
^{-1}s_{j}).
\end{equation*}
To see this, take $b\in B(a,r_{0})$. By (\ref{f}),
\begin{eqnarray*}
|b-y_{j}| &\leq &|b-a|+|a-y|+|y-y_{j}| \\
&\leq &3D^{-1}\dfrac{\eta }{100} \\
&\leq &\Vert \varphi _{w_{j}}^{^{\prime }}\Vert ^{-1}s_{j}.
\end{eqnarray*}
Choose $j_{2}\in \mathbb{N},j_{2}\geq j_{1}$, such that for all
$j\geq j_{2}$,
\begin{equation*} |y_{j}-y|\leq \dfrac{r_{0}}{2}.
\end{equation*}
If $a=(a^{1},..,a^{m},y^{m+1},..,y^{n})\in V_{y}$ let
$\widetilde{a}_{j}=(a^{1},..,a^{m},y_j^{m+1},..,y_j^{n})\in V_{j}$
and notice that \begin{equation*}| \widetilde{a}_{j}-a| \leq
|y-y_{j}|.
\end{equation*}
Then for $ j\geq j_{2}$ and $r_{1}=\dfrac{r_{0}}{2}$, by triangle
inequality,
\begin{equation}
\label{t1} \widetilde{a}_{j}\in B(y_{j},\Vert \varphi
_{w_{j}}^{^{\prime }}\Vert ^{-1}s_{j})
\end{equation}
and
\begin{equation}
\label{t2} B(\widetilde{a}_{j},r_{1})\subset B(a,r_{0}).
\end{equation}
Hence for $j_{\ast }\in \mathbb{N}$ big enough satisfying
\begin{equation*}
j_{\ast }\geq \max \{j_{0},j_{2}\} \text{ and } 2\Vert \varphi
_{w_{j_{\ast }}}^{^{\prime }}\Vert ^{-1}\dfrac{s_{j_{\ast
}}}{j_{\ast }}\leq r_{1}
\end{equation*}
we get,
\begin{enumerate}
\item $B(a,r_{0})\cap T_{j_{\ast }}=\emptyset$,
\item $\widetilde{a}_{j_{\ast }}\in V_{j_{\ast }}\cap B(y_{j_{\ast
}},\Vert \varphi _{w_{j_{\ast }}}^{^{\prime }}\Vert ^{-1}s_{j_{\ast
}})$,
\item $B(\widetilde{a}_{j_{\ast }},r_{1})\subset B(a,r_{0})$.
\end{enumerate}
Consequently
\begin{equation*}
B(\widetilde{a}_{j_{\ast }},2\Vert \varphi _{w_{j_{\ast
}}}^{^{\prime }}\Vert ^{-1}\dfrac{s_{j_{\ast }}}{j_{\ast }})\cap
T_{j_{\ast }}=\emptyset
\end{equation*}
which contradicts (\ref{int2}).

\item[T3] $T\subset F(E)$.

Let $t\in T$, then there exist some increasing function
$k(j):\mathbb{N\rightarrow N}$ and some $\alpha \in \mathcal{G}$
such that \begin{equation*}
 \alpha (k(j))\in T_{k(j)}\subset \Psi
_{k(j)}(\varphi _{w_{k(j)}}(E))=F_{k(j)}(E) \text{ and }\alpha
(k(j))\rightarrow t.
\end{equation*}
Therefore there exists a sequence $\{e_{j}\}_{j=1}^{\infty}\in E$
such that $F_{k(j)}(e_{j})=\alpha(k(j))$. Since the limit set $E$ is
compact there exists some subsequence of $\{e_{j}\}_{j=1}^{\infty}$
converging to some point $e\in E$. To simplify notation assume that
$e_{j}\rightarrow e$. Finally because the convergence
$F_{k(j)}\rightarrow F$ is uniform, we also deduce that
\begin{equation*}
\alpha (k(j))=F_{k(j)}(e_{j})\rightarrow F(e),
\end{equation*}
which implies that $t=F(e)$.

\end{description}

Properties (T2) and (T3) imply
\begin{equation*}
F^{-1}(B(y,D^{-1}\dfrac{\eta }{100})\cap V_{y})\subset
F^{-1}(T)\subset E.
\end{equation*}
Since $F^{-1}$ is conformal, this contradicts (\ref{conf}),
finishing the proof of Lemma \ref{confim}.
\end{proof}

\begin{proof}[Proof of Theorem \ref{cthm}] Let $x\in \mathbb{R}^{n}$ and $ 0<r <1$. For $I(x,r)\subset I^{\ast }, N \in \N$ as in (CIFS
2) we get
\begin{equation*}
I(x,r)=\{w_{1},...w_{m}\}\text{ for some }m\leq N \text{ and }
d(\varphi _{w_{i}}(E))\leq r\text{ for }i=1,..,m.
\end{equation*}
Applying Lemma \ref{confim} for $b=1$, as $r \geq d(\varphi
_{w_{1}}(E))$, there exist $z_{1}\in V_{x}$ and $l_{1}\geq 0$ such
that
\begin{equation*}
B(z_{1},l_{1})\subset B(x,r)\backslash \varphi _{w_{1}}(E)\text{ and
} l_{1}\geq a(1)r.
\end{equation*}
As
\begin{equation*}
r\geq d(\varphi _{w_{2}}(E))
\end{equation*}
we also get
\begin{equation*}
l_{1}\geq a(1)d(\varphi _{w_{2}}(E)).
\end{equation*}
Denote $a_{1}:=a(1)$. Again Lemma \ref{confim} implies that there
exist $z_{2}\in V_{x}$ and $l_{2}\geq 0$ satisfying
\begin{equation*}
B(z_{2},l_{2})\subset B(z_{1},l_{1})\backslash \varphi
_{w_{2}}(E)\subset B(x,r)\text{ and }l_{2}\geq a(a_{1})l_{1}.
\end{equation*}
As before
\begin{eqnarray*}
l_{2} &\geq &a(a_{1})a(1)r \\
&\geq &a(a_{1})a_{1}d(\varphi _{w_{_{3}}}(E)).
\end{eqnarray*}
In the same manner denote $a_{2}:=a(a_{1})a_{1}$. There exist $
z_{3}\in V_{x}$ and $l_{3}\geq 0$ such that
\begin{equation*}
B(z_{3},l_{3})\subset B(z_{2},l_{2})\backslash \varphi
_{w_{_{3}}}(E)
\end{equation*}
and
\begin{eqnarray*}
l_{3} &\geq &a(a_{2})l_{2} \\
&\geq &a(a_{2})a(a_{1})a_{1}r \\
&=&a(a_{2})a_{2}d(\varphi _{w_{_{4}}}(E)).
\end{eqnarray*}
Repeating the same arguments, after $m$ steps, we finally get that
there exist some $z_{m}\in V_{x}\cap B(x,r),l_{m}>0$ such that
\begin{equation*}
B(z_{m},l_{m})\subset B(z_{m-1},l_{m-1})\backslash \varphi
_{w_{m}}(E)
\end{equation*}
and
\begin{eqnarray*}
l_{m} &\geq &a(a_{m-1})l_{m-1} \\
&\geq &a(a_{m-1})...a(a_{1})a_{1}r.
\end{eqnarray*}
Therefore
\begin{eqnarray*}
B(z_{m},C^{\ast }r) &\subset &B(x,r)\backslash \underset{w\in
I(x,r)}{\bigcup }
\varphi _{w}(E) \\
&=&B(x,r)\backslash E
\end{eqnarray*}
where $C^{\ast }=a(a_{m-1})a_{m-1}=a(a_{m-1})...a(a_{1})a_{1}$ is a
constant depending only on the CIFS's initial parameters.
\end{proof}
\section{Geometric criteria for weak convergence}
We begin this section with an auxiliary result necessary to prove
Theorem \ref{msithm}.
\begin{thm}
\label{fst}Let $\mu $ be a finite Radon measure in $\Rn$ and
$K:\mathbb{R}^{n}\setminus \{0\}\rightarrow \mathbb{R}$ an
antisymmetric kernel satisfying (\ref{mg}) and (\ref{kg})
respectively.
\begin{enumerate}
\item The truncated singular integral
operators $T_{\varepsilon }$ associated to $\mu$ and $K$ converge
weakly in $\mathcal{X}_{Q}(\mathbb{R}^{n})$ if for any $V\in
TA(n,n-1)=\{V_{w}^{i}:i=1,..,n$ and $w\in \mathbb{R}^{n}\}$,
\begin{enumerate}\item $\mu(V)=0$,
\item there exists some positive number $a_{V} <1$ such that
\begin{equation}
\label{finitesum} \overset{\infty }{\underset{k=0}{\sum }}\mu
(S_{k}(a_{V},V))k<\infty,
\end{equation}
where $S_{k}(a_{V},V)=\{x\in \mathbb{R}^{n}:\overset{\infty
}{\underset{j=k+1}{ \sum }}a_{V}^{j}\leq d(x,V)< \overset{\infty
}{\underset{j=k}{\sum }} a_{V}^{j}\}.$
\end{enumerate}
\item The truncated singular integral operators $T_{\varepsilon }$,
associated to $\mu$ and $K$ converge weakly in
$\mathcal{X}_{B}(\mathbb{R}^{n})$ if for any sphere $C=S_x^R$,
centered at $x$ of radius $R$, \begin{enumerate}\item $\mu(C)=0$,
\item there exists some positive number $a_C <\textmd{min}\{1,R\}$
such that
\begin{equation}
\label{finitesum1} \overset{\infty }{\underset{k=0}{\sum }}\mu
(S_{k}(a_C,C))k<\infty ,
\end{equation}
where $S_{k}(a_C,C)=\{x\in B(x,R):\overset{\infty
}{\underset{j=k+1}{ \sum }}a_C^{j}\leq d(x,C)< \overset{\infty
}{\underset{j=k}{\sum }} a_C^{j}\}.$
\end{enumerate}
\end{enumerate}
\end{thm}
\begin{proof} We give the proof only for (i) since the proof of (ii) is almost identical. Denote $E=\textmd{spt}\mu$ and without loss of generality assume that $ E\subset
B(0,1/2)$ and $\mu(E) \leq 1$. Let
\begin{equation*}
f=\overset{l}{ \underset{i=1}{\sum }}a_{i}\chi _{Q_{i}}\text{
and } g=\overset{m}{\underset{j=1}{ \sum }}b_{j}\chi _{P_{j}}
\end{equation*}
where $a_{i},b_{j}\in \mathbb{R}$ and $Q_{i},P_{j}\in
Q(\mathbb{R}^{n})$. For $ 0<\delta  <\varepsilon $,
\begin{equation*}
\begin{split}
&\left| \int T_{\varepsilon }(f)(x)g(x)d\mu x-\int T_{\delta
}(f)(x)g(x)d\mu x\right|\\*[7pt] & = \left| \int \left( T_{\varepsilon }(f)(x)-T_{\delta }(f)(x\right) )%
\overset{m}{\underset{j=1}{\sum }}b_{j}\chi _{P_{j}}(x)d\mu
x\right|\\*[7pt]& = \left| \overset{m}{\underset{j=1}{\sum }}b_{j}\int_{P_{j\text{ }%
}}\int_{B(x,\varepsilon )\backslash B(x,\delta )}K(x-y)f(y)d\mu
yd\mu x\right|\\*[7pt]& \leq \overset{m}{\underset{j=1}{\sum }}\overset{l}{\underset{i=1}{\sum }}%
\left| b_{j}a_{i}\right| \left| \underset{\delta <\left| x-y\right|
<\varepsilon }{\int_{P_{j}}\int_{Q_{i}}}K(x-y)d\mu yd\mu x\right|.
\end{split}
\end{equation*}
By the antisymmetry of $K$ and Fubini's Theorem we have

\begin{equation*}
\begin{split}
&\left| \underset{\delta <\left| x-y\right| <\varepsilon
}{\int_{P_{j }}\int_{Q_{i}}}K(x-y)d\mu yd\mu x\right|\\*[7pt] & =
\left| \underset{\delta <\left| x-y\right| <\varepsilon
}{\int_{P_{j}}\int_{Q_{i}\cap P_{j}}} K(x-y)d\mu yd\mu
x+\underset{\delta <\left| x-y\right| <\varepsilon }{
\int_{P_{j}}\int_{Q_{i}\setminus P_{j}}}K(x-y)d\mu yd\mu
x\right|\\*[7pt]& \leq \left| \underset{\delta <\left| x-y\right|
<\varepsilon }{\int_{P_{j}\cap Q_{i}}\int_{Q_{i}\cap
P_{j}}}K(x-y)d\mu yd\mu x\right| +\left| \underset{\delta <\left|
x-y\right| <\varepsilon }{\int_{P_{j}\setminus Q_{i}}\int_{Q_{i}\cap
P_{j}}}K(x-y)d\mu yd\mu x\right|\\*[7pt] & + \left| \underset{\delta
<\left| x-y\right| <\varepsilon }{\int_{P_{j\text{ }}\setminus
Q_{i}}\int_{Q_{i}\setminus P_{j}}}K(x-y)d\mu yd\mu \right| +\left|
\underset{\delta <\left| x-y\right| <\varepsilon }{\int_{P_{j}\cap
Q_{i}}\int_{Q_{i}\setminus P_{j}}}K(x-y)d\mu yd\mu \right|\\*[7pt]
&\leq \underset{\delta <\left| x-y\right| <\varepsilon }{
\int_{Q_{i}}\int_{Q_{i}^{c}}}\left| K(x-y)\right| d\mu yd\mu
x+2\underset{ \delta <\left| x-y\right| <\varepsilon
}{\int_{P_{j}}\int_{P_{j}^{c}}}\left| K(x-y)\right| d\mu yd\mu x.
\end{split}
\end{equation*}
Therefore it is enough to show that for every $A\in
Q(\mathbb{R}^{n})$
\begin{equation}
\label{aac} \int_{A}\int_{A^{c}}\left| K(x-y)\right| d\mu yd\mu
x<\infty.
\end{equation}
Since $\mu(V)=0$ for every $V\in TA(n,n-1)$ instead of (\ref{aac})
it suffices to prove that
\begin{equation}
\label{aoac} \int_{A^\circ}\int_{A^{c}}\left| K(x-y)\right| d\mu
yd\mu x<\infty,
\end{equation}
for all $A\in Q(\mathbb{R}^{n})$. Let $G_{i}\in
TA(n,n-1),i=1,..,2n,$ be the hyperplanes that contain the $2n$ sides
of $A$. For any $x\in A^\circ\cap E$ and any $i=1,..,2n$ define the
following distance functions
\begin{equation*}
d_{i}(x)=d(x,G_{i}).
\end{equation*}
Let $N_{i}(x)>0,i=1,..,2n,$ be such that
\begin{equation*}
2^{N_{i}(x)}d_{i}(x)=1.
\end{equation*}
Hence if  $\left\lfloor N_{i}(x)\right\rfloor $ is the smallest
integer greater than $N_{i}(x)$
\begin{equation*}
\left\lfloor N_{i}(x)\right\rfloor \leq (\log 2)^{-1}\log
d_{i}(x)^{-1}+1.
\end{equation*}
Therefore
\begin{equation*}
E \setminus A\subset \overset{2n}{\underset{i=1}{\bigcup
}}\underset{j=1 }{\overset{\left\lfloor N_{i}(x)\right\rfloor
}{\bigcup }} B(x,2^{j}d_{i}(x))\setminus B(x,2^{j-1}d_{i}(x)),
\end{equation*}
and for all $x\in A^\circ\cap E$,
\begin{eqnarray*}
\int_{A^{c}}\left| K(x-y)\right| d\mu y &\leq &C_{K}\int_{\overset{2n}{%
\underset{i=1}{\bigcup }}\underset{j=1}{\overset{\left\lfloor
N_{i}(x)\right\rfloor }{\bigcup }}B(x,2^{j}d_{i}(x))\setminus
B(x,2^{j-1}d_{i}(x))}\left| x-y\right| ^{-(n-1)}d\mu y \\
&=&C_{K}\overset{2n}{\underset{i=1}{\sum }}\underset{j=1}{\overset{%
\left\lfloor N_{i}(x)\right\rfloor }{\sum }}\int_{B(x,2^{j}d_{i}(x))%
\setminus B(x,2^{j-1}d_{i}(x))}\left| x-y\right| ^{-(n-1)}d\mu y \\
&\leq &C_{K}\overset{2n}{\underset{i=1}{\sum }}\underset{j=1}{\overset{%
\left\lfloor N_{i}(x)\right\rfloor }{\sum }}\frac{\mu \left(
B(x,2^{j}d_{i}(x))\right) }{2^{-(n-1)}d_{i}(x)^{n-1}2^{j(n-1)}} \\
&\leq &C_{K}\overset{2n}{\underset{i=1}{\sum }}\underset{j=1}{\overset{%
\left\lfloor N_{i}(x)\right\rfloor }{\sum }}\frac{Cd_{i}(x)^{n-1}2^{j(n-1)}}{%
2^{-(n-1)}d_{i}(x)^{n-1}2^{j(n-1)}} \\
&\leq &C_{K}C2^{(n-1)}(\log 2)^{-1}(\overset{2n}{\underset{i=1}{\sum
}}\log d_{i}(x)^{-1}+2n).
\end{eqnarray*}
This leads to the following estimate
\begin{equation}
\label{fe} \int_{A^\circ}\int_{A^{c}}\left| K(x-y)\right| d\mu yd\mu
x \leq \frac{C_{K}C2^{(n-1)}}{\log
2}(\overset{2n}{\underset{i=1}{\sum }} \int_{A^\circ}\log
d_{i}(x)^{-1}d\mu x +2n).
\end{equation}
Notice that for $i=1,..,2n$, $A^\circ$ can be decomposed as
\begin{equation*}
A\subset \overset{\infty }{\underset{k=0}{\bigcup }}S_{k}(a_i,G_i)
\cup A'_i,
\end{equation*}
where $a_i=a_{G_i}$ and $A'_i=\{x \in A:d_i(x)>s_i=\sum_{j=0}^\infty
a_i^j\}$. Therefore
\begin{equation*}
\int_{A^\circ}\log d_{i}(x)^{-1}d\mu x\leq \overset{\infty
}{\underset{k=0}{\sum }} \int_{S_{k}(a_i,G_i)}\log d_{i}(x)^{-1}d\mu
x+\log s_i^{-1}.
\end{equation*}
For $x\in S_{k}(a_{G_i},G_i)$
\begin{equation*}
d_{i}(x)>\overset{\infty }{\underset{j=k+1}{\sum }}a_i^{j}=a_i^{k+1}
\dfrac{1}{1-a_i}
\end{equation*}
and
\begin{eqnarray*}
\log \frac{1}{d_{i}(x)} &\leq &\log \left(
\dfrac{1-a_i}{a_i^{k+1}}\right)
\\
&=&k\log \dfrac{1}{a_i}+\log \dfrac{1-a_i}{a_i}.
\end{eqnarray*}
Hence
\begin{equation}
\label{se} \int_{A^\circ}\log \frac{1}{d_{i}(x)}d\mu x\leq \log
\dfrac{1}{a_i}\overset{\infty }{ \underset{k=0}{\sum }}\mu
(S_{k}(a_i,G_i))k+\log \dfrac{1-a_i}{a_is_i}.
\end{equation}
Using (\ref{fe}) and (\ref{se}) we can estimate
\begin{equation*}
\begin{split}
&\int_{A^\circ}\int_{A^{c}}\left| K(x-y)\right| d\mu yd\mu x \leq \\
&\frac{C_{K}C2^{(n-1)}}{\log
2}\left(\overset{2n}{\underset{i=1}{\sum }}\log \dfrac{1
}{a_i}\overset{\infty }{\underset{k=0}{\sum }}\mu (S_{k}(a_i,G_i))k
+\overset{2n}{\underset{i=1}{\sum }}\log
\dfrac{1-a_i}{a_is_i}+2n\right).
\end{split}
\end{equation*}
Since, by (\ref{finitesum}), for $i=1,..,2n$
$$\sum_{k=0}^{\infty}\mu (S_{k}(a_i,G_{i}))k <\infty,$$ we have shown (\ref{aoac}) and the proof of Theorem \ref{fst}(i) is complete.
\end{proof}
We can now proceed in the proof of Theorem \ref{msithm}.
\begin{proof}[Proof of Theorem \ref{msithm}] Let $\textmd{spt}\mu=E$ and without loss of generality assume that $E\subset B(0,1/2)$. We start by proving (i). For $x\in \mathbb{R}^{n} ,r>0,i\in \{1,..,n\},q\in \mathbb{N}$ define
the following grids,
\begin{multline*}
Gr(x,r,i,q) =\{g\in A(x,r):g^{i}=x^{i}\text{ and for }1\leq j\leq n\
,j\neq i, \\
g^{j} =(x^{j}-\frac{r}{2})+\frac{r}{2q}(2k-1)\text{ for some }k=
1,..,q\}.
\end{multline*}
Since $E$ is $V^{i}-$directed porous for $i=1,..,n$, as an immediate
corollary of Definition \ref{mdf} there exists some $N\in \N,N\geq
2, $ such that for every $x\in \mathbb{R}^{n}$ and every $r>0$ there
exists some $y\in V_{x}^{i}\cap A(x,r)$ satisfying
\begin{equation}
\label{cubepor} A(y,rN^{-1})\subset A(x,r)\text{ }\setminus E.
\end{equation}
From (\ref{cubepor}) we also deduce that there exist some $M\in
\mathbb{N},M\geq 4,$ in fact we can even choose $M=2N$, such that
for every $x\in \mathbb{R}^{n}$, every $r>0$ and every $i=1,..,n$
there exists some $g_{(x,r,i)}\in Gr(x,r,i,M)$ such that
\begin{equation}
\label{grpor} A(g_{(x,r,i)},rM^{-1})\subset A(x,r)\backslash E.
\end{equation}
By Theorem \ref{fst} it is enough to show that for every $x\in
\mathbb{R}^{n}$ and every $i=1,..,n$
\begin{equation*}
\overset{\infty}{\underset{k=0}{\sum}}\mu(S_{k}(M^{-1},V_{x}^{i}))k<\infty.
\end{equation*}
Thus we need to estimate the measure $\mu$ of the strips
$V_{x}^{i}(2^{-1}M^{-k})$. The idea is to cover
$V_{x}^{i}(2^{-1}M^{-k})\cap E\cap A(x,1)$ with cubes from
$Q(\mathbb{R}^{n})$ of sidelength $M^{-k}$ with their centers in
$Gr(x,1,i,M^{k})$. The use of the specific grids allows us to count
the covering cubes easily. Note that in order to cover
$V_{x}^{i}(2^{-1}M^{-k})\cap A(x,1)$ with cubes in
$Q(\mathbb{R}^{n})$, of sidelength $M^{-k}$ and with centers in
$V_{x}^{i}$ we first cover $V_{x}^{i}\cap A(x,1)$ with cubes
$\{Q_{j}\}_{j \in J}$ in $Q(\mathbb{R}^{n-1})$. Then the required
cubes needed to cover $V_{x}^{i}(2^{-1}M^{-k})\cap A(x,1)$ will be
\begin{multline*}P_j=\{(y^1,..,y^i,..y^n)\in
\R^n:(y^1,..,y^{i-1},y^{i+1},..,y^n)\in Q_j \text{ and }\\
y_i\in[x^i-2^{-1}M^{-k},x^i+2^{-1}M^{-k})\}.\end{multline*} See
Figures \ref{fig1} and \ref{fig2} for an illustration.

For $x\in \mathbb{R}^{n},r>0$ and $i=1,..,n,$ denote
\begin{equation*}
Gr^{\ast }(x,r,i,M)=Gr(x,r,i,M)\setminus \{g_{(x,r,i)}\}.
\end{equation*}
Fix some $x\in \mathbb{R}^{n},r>0$ and $i=1,..,n,$ then by
(\ref{grpor})
\begin{equation*}
V_{i}^{x}(r(2M)^{-1})\cap E\cap A(x,r)\subset \underset{y\in
Gr^{\ast }(x,r,i,M)}{\bigcup }A(y,rM^{-1})
\end{equation*}
and
\begin{equation*}
\textmd{card}(Gr^{\ast }(x,r,i,M))=M^{n-1}-1.
\end{equation*}
Notice that the cardinality of the grid $Gr^{\ast }(x,r,i,M)$
depends only on its thickness, i.e. only on $M$.
\begin{figure}
\begin{minipage}[b]
{0.4\linewidth} % A minipage that covers half the page
\begin{center}
\includegraphics[width=7cm]{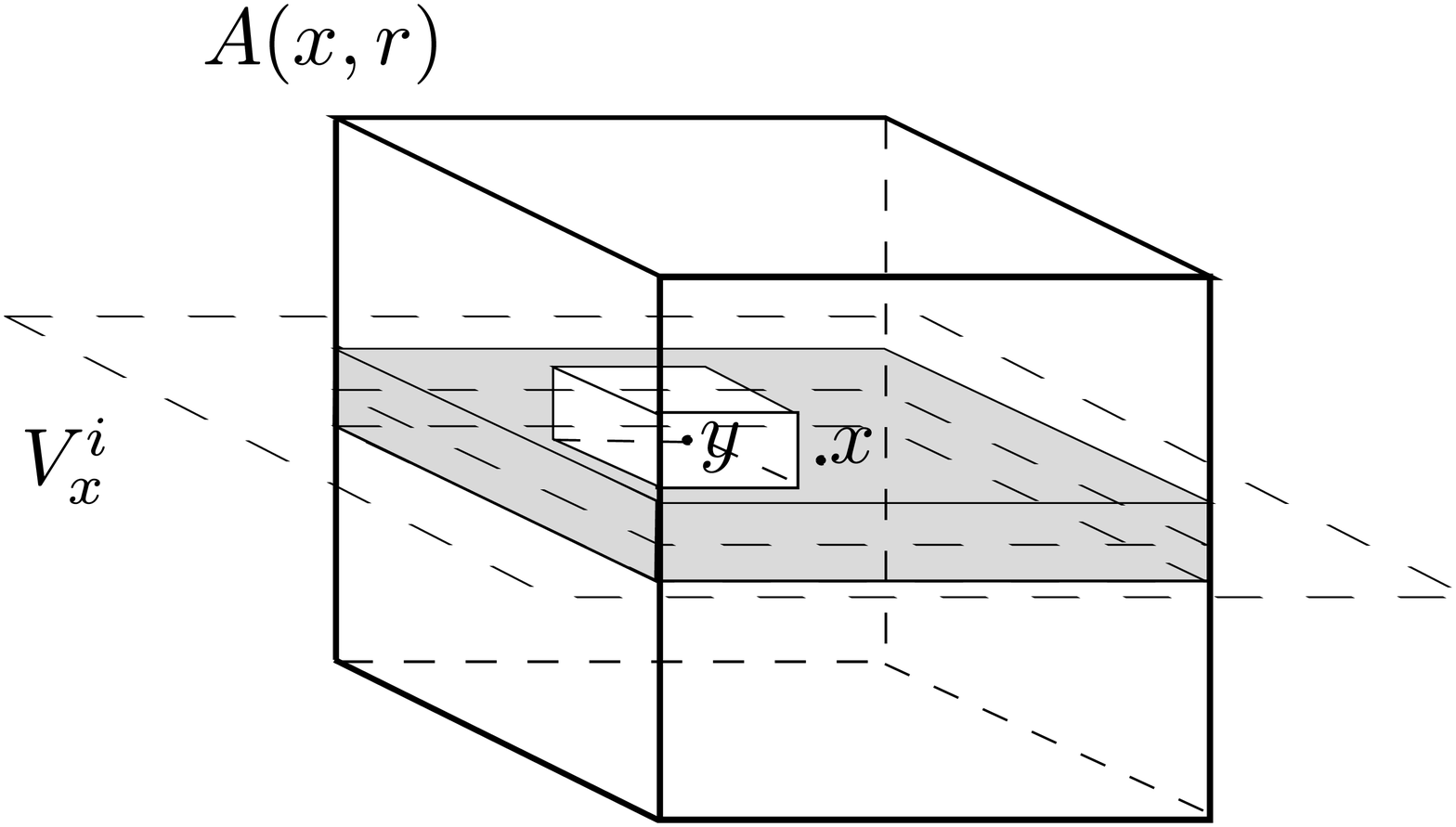}
\caption{}
\label{fig1}
\end{center}
\end{minipage}
\hspace{0.5cm}
\begin{minipage}[b]{0.4\linewidth}
\begin{center}
\includegraphics[width=7cm]{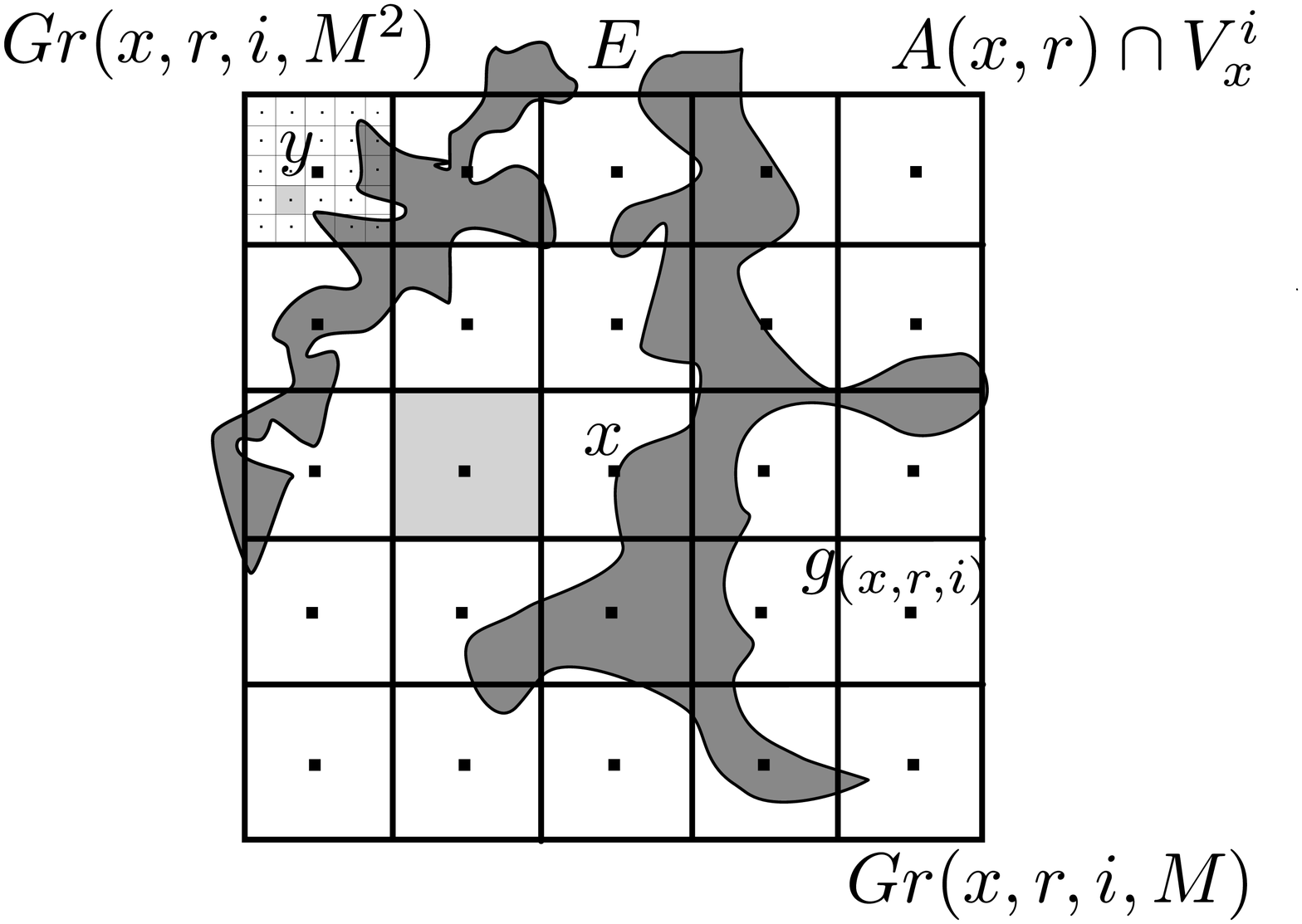}
\caption{} \label{fig2}
\end{center}
\end{minipage}
\end{figure}

In the same manner for $y\in Gr^{\ast }(x,r,i,M)$ the cubes
$A(y,rM^{-1})$ satisfy
\begin{equation*}
V_{i}^{x}(r2^{-1}M^{-2})\cap E\cap A(y,rM^{-1})\subset
\underset{h\in Gr^{\ast }(y,rM^{-1},i,M)}{\bigcup }A(h,rM^{-2}).
\end{equation*}
Therefore
\begin{equation*}
V_{i}^{x}(r2^{-1}M^{-2})\cap E\cap A(x,r)\subset \underset{\{h\in
Gr^{\ast }(y,rM^{-1},i,M):y\in Gr^{\ast }(x,r,i,M)\}}{\bigcup }
A(h,rM^{-2})
\end{equation*}
and
\begin{equation*}
\textmd{card}(\{h\in Gr^{\ast }(y,rM^{-1},i,M):y\in Gr^{\ast
}(x,r,i,M)\})=(M^{n-1}-1)^{2}.
\end{equation*}
Notice that
\begin{equation*}
\{h\in Gr^{\ast }(y,rM^{-1},i,M):y\in Gr^{\ast }(x,r,i,M)\}\subset
Gr(x,r,i,M^{2}).
\end{equation*}
Inductively we conclude that for all $x\in \mathbb{R}^{n},r>0,i\in
\{1,..,n\} $ and $k\in \mathbb{N}$ there exist sets of cubes
\begin{equation*}
Q^{k}(x,r,i)\subset Q(\mathbb{R}^{n}),
\end{equation*}
consisting of cubes $A(g,\frac{r}{M^{k}})$ with $g\in $
$Gr(x,r,i,M^{k})$ satisfying
\begin{enumerate}
\item $V_{i}^{x}(r2^{-1}M^{-k})\cap E\cap A(x,r)\subset \bigcup
\{Q:Q\in Q^{k}(x,r,i)\},$
\item $\textmd{card}(Q^{k}(x,r,i))=(M^{n-1}-1)^{k}$.
\end{enumerate}
Properties (i) and (ii) imply that for all $x\in
\mathbb{R}^{n},r>0,i=1,..,n$ and $k\in \mathbb{N}$
\begin{eqnarray*}
\mu(V_{x}^{i}(2^{-1}M^{-k})\cap A(x,1)) &\leq &\sum_{Q\in
Q^{k}(x,1,i)}\mu(Q)
\\*[3pt]
&\leq& \textmd{card}(Q^{k}(x,1,i))C(\sqrt{n}M^{-k})^{n-1} \\*[8pt]
&=&C(\sqrt{n})^{n-1}(1-M^{1-n})^{k}.
\end{eqnarray*}
For every $x\in \mathbb{R}^{n}$ and every $i=1,..,n$ there exist
$y_{(x,i)}^{1}$ and $y_{(x,i)}^{2}$ such that
\begin{equation*}
S_{k}(M^{-1},V_{x}^{i})=V_{y_{(x,i)}^{1}}^{i}(2^{-1}M^{-k})\cup
V_{y_{(x,i)}^{2}}^{i}(2^{-1}M^{-k})
\end{equation*}
and
\begin{equation*}
S_{k}(M^{-1},V_{x}^{i})\cap E\subset A(y_{(x,i)}^{1},1) \cup
A(y_{(x,i)}^{2},1).
\end{equation*}
Therefore we deduce that
\begin{equation*}
\begin{split}
\sum_{k=0}^{\infty }\mu(S_{k}(M^{-1},V_{x}^{i}))k
&=\underset{k=0}{\overset{\infty }{\sum }}
\mu(V_{y_{(x,i)}^{1}}^{i}(2^{-1}M^{-k})\cap A(y_{(x,i)}^{1},1))k \\
&+\underset{k=0}{\overset{\infty }{\sum }}\mu(V_{y_{(x,i)}^{2}}^{i}(2^{-1}M^{-k}) \cap A(y_{(x,i)}^{2},1))k \\
&\leq 2C(\sqrt{n})^{n-1}\sum_{k=0}^\infty (1-M^{1-n})^{k}k.
\end{split}
\end{equation*}
This concludes the proof of (i) since $$\sum_{k=0}^\infty
(1-M^{1-n})^{k}k<\infty.$$ For the proof of (ii) notice that since
$E$ is $V$-directed porous for all $V\in G(n,n-1)$ we can define the
function, $\Theta:G(n,n-1)\rightarrow (0,1)$, as
$$\Theta(V)= c(V)$$ where $c(V)$
are the numbers appearing in Definition \ref{mdf}. By compactness of
$G(n,n-1)$, see e.g.\cite{M2}, and continuity of $\Theta$, we deduce
that that $\Theta$ attains some minimal value $c$ depending only on
the set $E$. Using this observation, Theorem \ref{fst} (ii) and
exactly the same arguments as in (i), adapted to spheres, we obtain
(ii).
\end{proof}
\emph{Acknowledgement.} I express my gratitude to my advisor,
Professor Pertti Mattila, for many ideas, discussions, suggestions
and for repeated proof reading. I would also like to thank Antti
K\"{a}enm\"{a}ki, for many valuable comments and for pointing out
that the open set condition is the only extra assumption needed for
the CIFS discussed in Section 2.

\vspace{1cm}
\begin{footnotesize}
{\sc Department of Mathematics and Statistics,
P.O. Box 68,  FI-00014 University of Helsinki, Finland,}\\
\emph{E-mail address:} \verb"vasileios.chousionis@helsinki.fi"
\end{footnotesize}

\end{document}